\documentclass[12pt]{amsart}
\usepackage{graphicx,amsthm,amssymb,amsmath,setspace,ulsy,tikz-cd,amsfonts,mathrsfs,accents,fancyhdr,nameref,multicol,wasysym} 
\usepackage[all]{xy}
\usepackage[shortlabels]{enumitem}
\usepackage[square,numbers]{natbib}

\usepackage{tikz}
\usetikzlibrary{decorations.markings,arrows,hobby,knots,positioning,backgrounds}
\newcommand{\midarrow}{\tikz \draw[-triangle 90] (0,0) -- +(.1,0);}

\title{Complexity of virtual multistrings}
\author{David Freund}

\newtheorem{thm}{Theorem}[section]
\newtheorem{lem}[thm]{Lemma}
\newtheorem{cor}[thm]{Corollary}
\newtheorem{prop}[thm]{Proposition}

\theoremstyle{remark}
\newtheorem{rem}[thm]{Remark}
\newtheorem{ex}[thm]{Example}

\newcommand{\es}{\varnothing} 
\newcommand{\A}{\alpha}
\newcommand{\B}{\beta}

\newcommand{\G}{\gamma}

\newcommand{\set}[1]{\left\{#1\right\}} 
\newcommand{\wt}[1]{\widetilde{#1}} 

\newcommand{\Rmoves}[1]{
 \begin{tikzpicture}[#1]
	\begin{knot}[xshift=-2cm,clip width=0,end tolerance=.1cm]
		\strand[thick]
		(135:1) to[out=-45,in=90] 
		(0,0) to[out=-90,in=45]
		(225:1);
	\end{knot}
	\begin{knot}[xshift=2cm,clip width=0,end tolerance=.1cm]
		\strand[thick]
		(135:1) to[out=-60,in=-90,looseness=2]
		(.5,0) to[out=90,in=60,looseness=2]
		(225:1);
	\end{knot}
	\draw[dashed] (-2,0) circle (1);
	\draw[dashed] (2,0) circle (1);
	\draw[<->,thick] (-.5,0) -- node[auto] {1} (.5,0);
	
	\begin{scope}[xshift=7cm]
	\begin{knot}[xshift=-2cm,end tolerance=.1cm]
		\strand[thick]
		(135:1) to[out=-45,in=45] 
		(225:1);
		\strand[thick]
		(45:1) to[out=-135,in=135]
		(-45:1);
	\end{knot}
	
	\begin{knot}[xshift=2cm,clip width=0,end tolerance=.1cm]
	\strand[thick]
		(135:1) to[out=-20,in=90,looseness=.75]	
	  (0.5,0) to[out=-90,in=20,looseness=.75]
		(225:1);
	\strand[thick]
		(45:1) to[out=-160,in=90,looseness=.75]
		(-.5,0) to[out=-90,in=160,looseness=.75]
		(-45:1);
	\end{knot}
		
	\draw[dashed] (-2,0) circle (1);
	\draw[dashed] (2,0) circle (1);
	\draw[<->,thick] (-.5,0) -- node[auto] {2} (.5,0);
	\end{scope}
	
	\begin{scope}[xshift=14cm]
	\begin{knot}[xshift=-2cm,clip width=0,end tolerance=.1cm]
		\strand[thick]
			(135:1) to (-45:1);
		\strand[thick]
			(-135:1) to (45:1);
		\strand[thick]
			(180:1) to[out=-20,in=180,looseness=1]
			(0,.75) to[out=0,in=200,looseness=1]
			(0:1);
	\end{knot}
	
	\begin{knot}[xshift=2cm,clip width=0,end tolerance=.1cm]
		\strand[thick]
			(135:1) to (-45:1);
		\strand[thick]
			(-135:1) to (45:1);
		\strand[thick]
			(180:1) to[out=20,in=180,looseness=1]
			(0,-.75) to[out=0,in=160,looseness=1]
			(0:1);
	\end{knot}

	\draw[dashed] (-2,0) circle (1);
	\draw[dashed] (2,0) circle (1);
	\draw[<->,thick] (-.5,0) -- node[auto] {3} (.5,0);
	\end{scope}
\end{tikzpicture}}

\begin{document}

\begin{abstract} 
A {\em virtual $n$-string} $\alpha$ is a collection of $n$ oriented smooth generic loops on a surface $M$. A {\it stabilization} of $\alpha$ is a surgery that results in attaching a handle to $M$ along disks avoiding $\alpha$, and the inverse operation is a {\em destabilization} of $\alpha$. We consider virtual $n$-strings up to {\em virtual homotopy}, sequences of stabilizations, destabilizations, and homotopies of $\alpha$.

Recently, Cahn~\cite{Cahn} proved that any virtual $1$-string can be virtually homotoped to a genus-minimal and crossing-minimal representative by monotonically decreasing both genus and the number of self-intersections. We generalize her result to the case of connected non-parallel $n$-strings.

Cahn~\cite{Cahn} also proved that any two crossing-irreducible representatives of a virtual $1$-string are related by Type 3 moves, stabilizations, and destabilizations. Kadokami~\cite{Kadokami} claimed that this held for virtual $n$-strings in general, but Gibson~\cite{Gibson} found a counterexample for $5$-strings. We show that Kadokami's statement holds for connected non-parallel $n$-strings and exhibit a counterexample for $3$-strings.
\end{abstract}
\maketitle

\leftline {\em \Small 2010 Mathematics Subject Classification. Primary: 57M99}

\leftline{\em \Small Keywords: virtual homotopy, virtual strings, curves on surfaces}

\section{Introduction}

Throughout this paper, we work in the smooth ($C^\infty$) category. A {\it virtual $n$-string} $\A$ is a collection of $n$ oriented smooth generic loops on a closed oriented (not necessarily connected) surface $M$. Each loop is a {\it component} of $\A$ and we assume each connected component of $M$ contains a component of $\A$. We regard a homotopy of $\A$ as a composition of isotopies and flat Reidemeister moves (depicted in Figure~\ref{fig:moves}) and denote the homotopy class of $\A$ by $[\A]$.

\begin{figure}
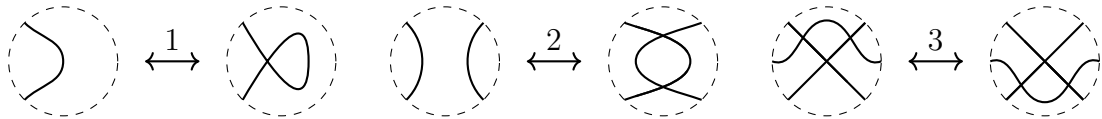

\Rmoves{scale=0.725}
\caption{Flat Reidemeister moves.}
\label{fig:moves}
\end{figure}

We consider equivalence classes of virtual $n$-strings 
up to {\it virtual homotopy}, a composition of stabilizations, destabilizations, and homotopies of $\A$. Let $[\A]_V$ denote the virtual homotopy class of $\A$. The virtual homotopy class $[\A]_V$ is also called a {\em flat virtual link} (e.g.,~\cite{Cahn,Henrich}), a projected virtual link (e.g.,~\cite{Kadokami}), the shadow of a virtual link (e.g.,~\cite{Kauffman}), or the universe of a virtual link (e.g.,~\cite{Carter}). The terminology of {\it virtual strings} (i.e., virtual 1-strings) was introduced by Turaev~\cite{Turaev} who used a combinatorial model of these objects.


In what follows, we consider two significant quantities related to a virtual $n$-string $\A$: the genus of the supporting surface and the number of double points ({\it crossings}). Each of these quantities measures a type of complexity for the virtual $n$-string and we are interested in understanding how these quantities vary under virtual homotopies.

Borrowing terminology from Cahn~\cite{Cahn}, we say that a virtual $n$-string $\A$ is {\it crossing-reducible} if there is a virtual $n$-string $\A'$, related to $\A$ by stabilizations, destabilizations, and Type~3 moves, such that a crossing-reducing Type~1 or Type~2 move may be applied to $\A'$. A {\it crossing-irreducible} $n$-string is one which is not crossing-reducible, and a {\it crossing-minimal} $n$-string exhibits the minimal number of crossings in its virtual homotopy class.

A virtual $n$-string $\A$ on a surface $M$ is {\it genus-reducible} if there is a homotopically nontrivial simple loop $\G$ on $M$ disjoint from $\A$. Cutting along $\G$ and capping off the resulting boundary with disks, we obtain a virtually homotopic $n$-string $\A'$ on a surface of smaller genus. A {\it genus-irreducible} $n$-string is one which is not genus-reducible, and a {\it genus-minimal} $n$-string exhibits the minimal genus of any $n$-string in the associated virtual homotopy class. By convention, we say that the genus of a non-connected surface is the sum of the genera of the connected components.

We say a virtual $n$-string $\A=\set{L_1,\dots,L_n}$ is {\it non-parallel} if, under any sequence of virtual homotopies, no pair of component curves are powers of parallel curves.
Deferring the definition of a connected multistring to Section 2, our main result is the following:

\begin{thm}\label{thm:complexity} Let $\A$ be a connected non-parallel virtual $n$-string. Then there is a genus-minimal and crossing-minimal $n$-string virtually homotopic to $\A$, unique up to isotopies and Type 3 moves. Moreover, such an $n$-string can be obtained so that neither the genus nor the number of intersection points increase during the virtual homotopy.\end{thm}



Theorem~\ref{thm:complexity} follows by considering a collection of established results, generalizing appropriately, and applying them in succession. In Section 2, we show that genus can be decreased monotonically for all virtual $n$-strings. In Section 3, we show that crossing-irreducibility implies crossing-minimality for disjoint unions of connected non-parallel multistrings. Then we build on these results to prove Theorem~\ref{thm:complexity} in Section 4 and conclude with brief remarks in Section 5.

\section{Genus-Irreducible Strings}

Given a virtual $n$-string $\A$, we first focus on finding a virtually-homotopic genus-irreducible string $\A'$. In contrast with crossing-irreducibility, due to the possibility of parallel curves, we are always able to obtain a genus-irreducible representative by monotonically decreasing genus. As noted by Cahn (see Theorem 11.2 in~\cite{Cahn}), Ilyutko, Manturov, and Nikonov proved the following (using our terminology):

\begin{thm}[Theorem 1.2 in~\cite{Manturov}]\label{thm:IMN'} Let $\A$ and $\A'$ be two virtually-homotopic genus-irreducible $1$-strings supported on surfaces $M$ and $M'$, respectively. Then there is an orientation-preserving automorphism $\phi: M\to M'$ such that $\phi(\A)$ is homotopic to $\A'$ on $M'$.
\end{thm}

More generally, the same proof applies to virtual $n$-strings.

\begin{thm}\label{thm:IMNfull} Let $\A$ and $\A'$ be two virtually-homotopic genus-irreducible $n$-strings supported on surfaces $M$ and $M'$, respectively. Then there is an orientation-preserving automorphism $\phi: M\to M'$ such that $\phi(\A)$ is homotopic to $\A'$ on $M'$.\end{thm}

Consequently, for a virtual $n$-string $\A$, there is a unique (up to homotopy) genus-irreducible $n$-string in $[\A]_V$.
Hence genus-irreducible $n$-strings are genus-minimal. Stated differently, any local minimum in genus is a global minimum for virtual $n$-strings. 

The proof of Theorem~\ref{thm:IMNfull} depends upon a generalization of Lemma 1.1 from~\cite{Manturov} to virtual $n$-strings, which is established similarly and we record here for further use:

\begin{lem} \label{lem:IMN} Let $\A,\A'$ be two virtual $n$-strings realized on a common surface $M$, where $\A'$ is obtained from $\A$ by a sequence of decreasing homotopies (i.e., a composition of decreasing Type 1 and 2 moves, Type 3 moves, and isotopies). Let $\G$ be a homotopically nontrivial closed curve in $M$ such that $\G\cap \A = \es$. Then $\G$ can be isotoped simultaneously with the homotopies of $\A$ such that the curves $\A_t$ and $\G_t$ will not intersect for all $t$. In particular, there is a curve $\G'$ isotopic to $\G$ such that $\G'\cap \A'=\es$.\end{lem}

As an immediate consequence of Theorem~\ref{thm:IMNfull}, we obtain:

\begin{cor}\label{cor:IMNdiscon}
Let $\A$ be a virtual $n$-string on a non-connected surface $M$, $n\geq 2$. Then any genus-irreducible $n$-string in $[\A]_V$ is supported on a non-connected surface.
\end{cor}

Thus, if $[\A]_V$ has a genus-irreducible representative on a non-connected surface, it is not possible to cleverly realize $[\A]_V$ on a connected surface of the same genus. We define a virtual $n$-string $\A$ to be {\it connected} if a genus-irreducible representative of $[\A]_V$ is supported on a connected surface. 

We now show that determining whether a connected virtual $n$-string is non-parallel can be established by considering a minimal genus representative. 

\begin{prop} \label{prop:parallel}
Let $\A$ be a connected virtual $n$-string and $\A'$ a genus-minimal representative of $[\A]_V$. Then $\A$ is non-parallel if and only if no component curves of $\A'$ are homotopic to powers of parallel curves.
\end{prop}

\begin{proof}
If $\A$ is non-parallel, then the result follows by definition. To show the converse, we assume that two components of $\A$, $L_1$ and $L_2$, are homotopic to powers of parallel curves and that $\A$ is genus-reducible.

Let $M$ be the surface supporting $\A$. Since $\A$ is genus-reducible, there is a homotopic
$n$-string $\A'$ and a destabilizing curve $\G$ disjoint from $\A'$. Let $L_1',L_2'$ be the deformations of $L_1,L_2$ under this homotopy.

As $L_1',L_2'$ are homotopic to powers of parallel curves, we may assume that they are contained in the same connected component of $M\setminus \G$. 
Destabilizing $M$ along $\G$, we obtain a new virtual $n$-string $\A''$ in which the components corresponding to $L_1',L_2'$ are still powers of parallel curves. 

By repeating this process, we obtain a genus-irreducible representative $\B$ of $[\A]_V$ on which two component curves are powers of parallel curves. By Theorem~\ref{thm:IMNfull}, $\B$ is a genus-minimal representative and every other such representative is homotopic to it. 
\end{proof}

\begin{rem} In the proof, we implicitly used connectedness by assuming that destabilizing would not produce a non-connected surface. Consequently, for non-connected $n$-strings, this result can fail.
\end{rem}

\begin{ex} Consider the 2-string $\A$ consisting of two circles on a sphere, one on each hemisphere. Then $[\A]_V$ is parallel. However, the equator of the sphere $\G$ is a destabilizing curve, and destabilizing along $\G$ produces two spheres, each supporting a circle. The resulting 2-string $\A'$ is genus-minimal and the component curves are not homotopic to parallel curves, forming a counterexample to Proposition~\ref{prop:parallel} for non-connected multistrings.
\end{ex}

As an immediate consequence of Proposition~\ref{prop:parallel}, we can show that a $2$-string is connected and non-parallel by considering a genus-irreducible representative.

\begin{ex} Consider $\A$, the double Kishino double sitting on a genus 4 surface as in Figure~\ref{fig:DKD}. We show that the 2-string is genus-irreducible. It follows that the double Kishino double is connected and non-parallel; hence that the minimal number of intersection points between the component Kishino doubles is 4 by Theorem~\ref{thm:complexity}.
\begin{figure}
\includegraphics[scale=.4]{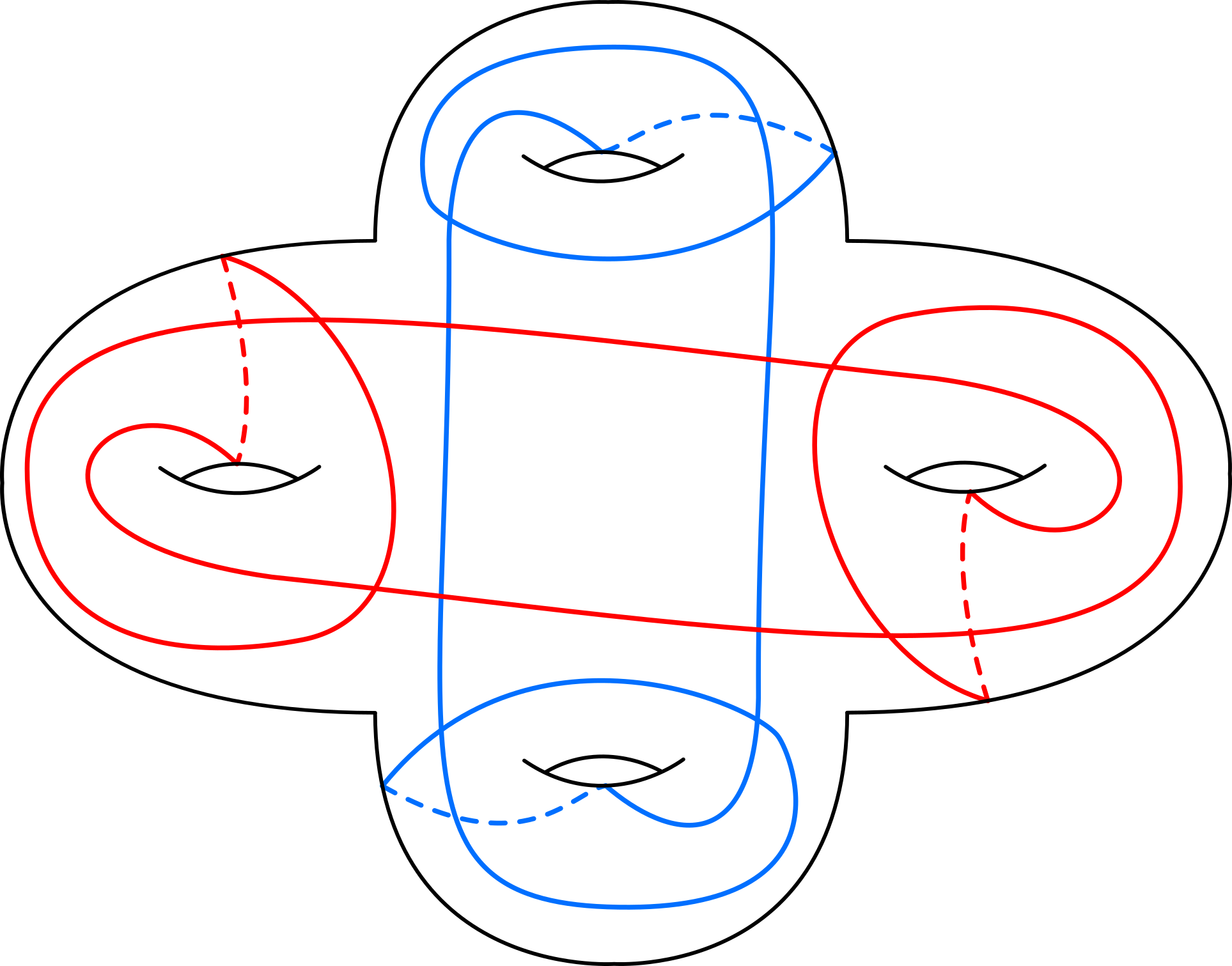}
\caption{Double Kishino double on a genus-4 surface.}
\label{fig:DKD}
\end{figure}

Suppose $\A'$ is homotopic to $\A$ and $\A'$ is genus-reducible. Let $\G$ be the destabilizing curve. Since the component Kishino doubles are not parallel, Theorem~\ref{thm:Hass} applies and so, by Lemma~\ref{lem:IMN}, $\G$ can be isotoped to a destabilizing curve for $\A$. However, no such curve exists for $\A$, and so $\A$ is genus-irreducible as claimed.
\end{ex}

\section{Kadokami's Statement}

The original statement for the uniqueness of crossing-irreducible virtual $n$-strings was first proposed by Kadokami~\cite{Kadokami}. He claimed that any two virtually-homotopic crossing-irreducible virtual $n$-strings were related by a (possibly empty) sequence of Type 3 moves, isotopies, stabilizations, and destabilizations. However, Gibson~\cite{Gibson} exhibited a counterexample to Kadokami's statement for the case $n=5$. Later, Cahn~\cite{Cahn} proved that Kadokami's statement holds for $1$-strings by combining results of Hass and Scott with Theorem~\ref{thm:IMN'}.

In this section, we show that Kadokami's statement holds more generally and give a counterexample for $3$-strings.
More formally, we prove the following:

\begin{thm} \label{thm:Kadokami} Let $\A$ be a connected non-parallel $n$-string. Any two crossing-irreducible representatives of $[\A]_V$ are related by a (possibly empty) sequence of Type 3 moves, isotopies, stabilizations, and destabilizations.\end{thm}


As a consequence, for such $n$-strings, a local minimum in the number of crossings is necessarily a global minimum. That is, crossing-irreducibility implies crossing-minimality for 
connected non-parallel multistrings.


\subsection{The Hass--Scott theorem}

Hass and Scott~\cite{Hass} define a curve to be an immersed (not necessarily connected) 1-dimensional manifold in a compact oriented surface $M$, where no pair of components are homotopic to powers of parallel curves. In particular, Hass and Scott's curves include non-parallel virtual $n$-strings and so we phrase their results with this terminology.

\begin{thm}[Hass--Scott, Theorem 2.2 in~\cite{Hass}] \label{thm:Hass} Let $\A_0$ be a non-parallel $n$-string 
which does not minimize the number of crossings in $[\A_0]$. Then there is a homotopy $\A_t$ from $\A_0$ to $\A_1$, which realizes the minimal number of crossings, such that the number of crossings of the curve $\A_t$ is non-increasing with $t$. Moreover, $\A_t$ is a regular homotopy except for a finite number of times when a small loop 
 shrinks to a point.\end{thm}

\begin{cor}[Corollary 2.3 in~\cite{Hass}] \label{cor:Hass} Let $\A_0$ and $\A_1$ be homotopic non-parallel virtual $n$-strings, each with $k$ crossings. There is a homotopy $\A_t$ from $\A_0$ to $\A_1$ with the property that each curve $\A_t$ has at most $k$ crossings.\end{cor}




\subsection{Proof of Theorem~\ref{thm:Kadokami}}

\begin{proof} Suppose that $\A,\A'$ are virtually homotopic crossing-irreducible $n$-strings that are non-parallel. Let $M,M'$ be the respective supporting surfaces. Then $\A$ and $\A'$ have the minimal number of crossings in their respective homotopy classes on $M,M'$. By Corollary~\ref{cor:Hass}, any other crossing minimal representative is related to $\A,\A'$ by Type 3 moves. We claim that we may further assume $\A$ and $\A'$ are genus-irreducible representatives.

Suppose $\A$ is genus-reducible. Then there is a $n$-string $\A''$ homotopic to $\A$ and a homotopically nontrivial closed curve $\G$ such that destabilizing $\A''$ along $\G$ produces a smaller genus surface. Since $[\A]_V$ is non-parallel, Lemma~\ref{lem:IMN} allows $\G$ to be isotoped to a destabilizing curve for $\A$. Hence we obtain a smaller genus representation of $[\A]_V$ without changing the number of crossings. Thus we assume that $\A,\A'$ are genus-irreducible.

By Theorem~\ref{thm:IMNfull}, there is an orientation-preserving automorphism $\phi: M\to M'$ such that $\phi(\A)$ is homotopic to $\A'$ on $M'$. By Corollary~\ref{cor:Hass}, we may assume that the number of crossings does not increase during this homotopy. As $\A,\A'$ are crossing-irreducible, the number cannot decrease either. Hence $\phi(\A)$ and $\A'$ are related by either a regular isotopy or a sequence of Type 3 moves.
\end{proof}

\subsection{Counterexample for $3$-strings} 

Gibson's counterexample to the general version of Kadokami's statement relied on having components which were powers of parallel curves and interchanging them. As this is the only obstruction to Hass and Scott's results (see~\cite{Hass}), we use a similar method for virtual $3$-strings.

Consider the virtual 3-strings depicted in Figure~\ref{fig:CE3}. The red and green components are powers of parallel curves on the two-holed torus. The 3-strings are clearly homotopic by interchanging the red and green components. Since the multistring based matrix of each 3-string is primitive (see Example 3.8 in~\cite{MBM}), they are crossing-minimal diagrams. 
We claim that these 3-strings are not related by a sequence of Type 3 moves and isotopies.

\begin{figure}[ht!]
\begin{tikzpicture}[scale=1.5]
\begin{scope}[very thick,every node/.style={sloped,allow upside down}]
\begin{scope}[inner sep=0pt,outer sep=0pt]
		\node at (1.35,-.5) (RB){};
		\node at (1.35,.5) (RT){};
		\node at (.5,1.35) (TR){};
		\node at (-.5,1.35) (TL){};
		\node at (-1.35,.5) (LT){};
		\node at (-1.35,-.5) (LB){};
		\node at (-.5,-1.35) (BL){};
		\node at (.5,-1.35) (BR){};
\end{scope}

\draw[postaction={decorate},decoration={
    markings,
    mark=at position .0625 with {\arrow{stealth}},
    mark=at position .1875 with {\arrowreversed{stealth}},
		mark=at position .1975 with {\arrowreversed{stealth}},
    mark=at position .3175 with {\arrowreversed{stealth}},
    mark=at position .4475 with {\arrowreversed{latex}}, 
		mark=at position .4575 with {\arrowreversed{latex}}, 
    mark=at position .5775 with {\arrowreversed{latex}}, 
    mark=at position .6875 with {\arrow{latex}}, 
		mark=at position .6975 with {\arrow{latex}}, 
		mark=at position .8175 with {\arrow{latex}}, 
    mark=at position .9375 with {\arrow{stealth}},
		mark=at position .9475 with {\arrow{stealth}},
    }
  ] (RB.west)--(RT.west)--(TR.south)--(TL.south)--(LT.east)--(LB.east)--(BL.north)--(BR.north)--(RB.west)--cycle;

\foreach\k in {1,2,3,4,5}
{
\pgfmathsetmacro\result{\k * .1666} 

\draw[draw=none] (RB)-- node[pos=\result] (Q1\k){} (RT); 
\draw[draw=none] (RT)-- node[pos=\result] (Q2\k){} (TR);
\draw[draw=none] (TR)-- node[pos=\result] (Q3\k){} (TL);
\draw[draw=none] (TL)-- node[pos=\result] (Q4\k){} (LT);
\draw[draw=none] (LT)-- node[pos=\result] (Q5\k){} (LB);
\draw[draw=none] (LB)-- node[pos=\result] (Q6\k){} (BL);
\draw[draw=none] (BL)-- node[pos=\result] (Q7\k){} (BR);
\draw[draw=none] (BR)-- node[pos=\result] (Q8\k){} (RB);
}

\begin{scope}[on background layer]
\draw[blue, very thick] (Q33.center) -- (0,-.2) to[out=-90,in=90] (-.4,-.45) to[out=-90,in=90] (0,-.7) -- (Q73.center);
\draw[blue, very thick] (Q13.center) -- node[pos=0.75]{\midarrow} (Q53.center);

\draw[red, very thick] (Q65.center) to[out=45,in=135] node[pos=0.25]{\midarrow} (Q81.center);
\draw[red, very thick] (Q25.center) to[out=225,in=-45] (Q41.center);

\draw[black!30!green, very thick] (Q62.center) to[out=45,in=135] node[pos=0.25]{\midarrow} (Q83.center); 
\draw[black!30!green, very thick] (Q63.center) to[out=45,in=135] (Q84.center);

\draw[black!30!green, very thick] (Q23.center) to[out=225,in=-45] (Q43.center); 
\draw[black!30!green, very thick] (Q22.center) to[out=225,in=-45] (Q44.center);
\end{scope}
\end{scope}

\begin{scope}[xshift=3cm]

\begin{scope}[very thick,every node/.style={sloped,allow upside down}]
\begin{scope}[inner sep=0pt,outer sep=0pt]
		\node at (1.35,-.5) (RB){};
		\node at (1.35,.5) (RT){};
		\node at (.5,1.35) (TR){};
		\node at (-.5,1.35) (TL){};
		\node at (-1.35,.5) (LT){};
		\node at (-1.35,-.5) (LB){};
		\node at (-.5,-1.35) (BL){};
		\node at (.5,-1.35) (BR){};
\end{scope}

\draw[postaction={decorate},decoration={
    markings,
    mark=at position .0625 with {\arrow{stealth}},
    mark=at position .1875 with {\arrowreversed{stealth}},
		mark=at position .1975 with {\arrowreversed{stealth}},
    mark=at position .3175 with {\arrowreversed{stealth}},
    mark=at position .4475 with {\arrowreversed{latex}}, 
		mark=at position .4575 with {\arrowreversed{latex}}, 
    mark=at position .5775 with {\arrowreversed{latex}}, 
    mark=at position .6875 with {\arrow{latex}}, 
		mark=at position .6975 with {\arrow{latex}}, 
		mark=at position .8175 with {\arrow{latex}}, 
    mark=at position .9375 with {\arrow{stealth}},
		mark=at position .9475 with {\arrow{stealth}},
    }
  ] (RB.west)--(RT.west)--(TR.south)--(TL.south)--(LT.east)--(LB.east)--(BL.north)--(BR.north)--(RB.west)--cycle;

\foreach\k in {1,2,3,4,5}
{
\pgfmathsetmacro\result{\k * .1666} 

\draw[draw=none] (RB)-- node[pos=\result] (Q1\k){} (RT); 
\draw[draw=none] (RT)-- node[pos=\result] (Q2\k){} (TR);
\draw[draw=none] (TR)-- node[pos=\result] (Q3\k){} (TL);
\draw[draw=none] (TL)-- node[pos=\result] (Q4\k){} (LT);
\draw[draw=none] (LT)-- node[pos=\result] (Q5\k){} (LB);
\draw[draw=none] (LB)-- node[pos=\result] (Q6\k){} (BL);
\draw[draw=none] (BL)-- node[pos=\result] (Q7\k){} (BR);
\draw[draw=none] (BR)-- node[pos=\result] (Q8\k){} (RB);
}

\begin{scope}[on background layer]
\draw[blue, very thick] (Q33.center) -- (0,-.6) to[out=-90,in=90] (-.35,-.85) to[out=-90,in=90] (0,-1.1) -- (Q73.center);
\draw[blue, very thick] (Q13.center) -- node[pos=0.75]{\midarrow} (Q53.center);

\draw[red, very thick] (Q62.center) to[out=45,in=135] node[pos=0.25]{\midarrow} (Q84.center);
\draw[red, very thick] (Q22.center) to[out=225,in=-45] (Q44.center);

\draw[black!30!green, very thick] (Q64.center) to[out=45,in=135] node[pos=0.25]{\midarrow} (Q81.center); 
\draw[black!30!green, very thick] (Q65.center) to[out=45,in=135] (Q82.center);

\draw[black!30!green, very thick] (Q25.center) to[out=225,in=-45] (Q41.center); 
\draw[black!30!green, very thick] (Q24.center) to[out=225,in=-45] (Q42.center);
\end{scope}
\end{scope}\end{scope}
\end{tikzpicture}
\caption{Homotopic crossing-irreducible 3-strings on a two-holed torus not related by Type 3 moves.}
\label{fig:CE3}
\end{figure}
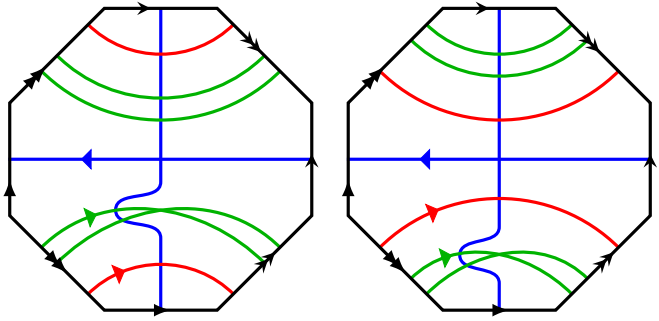

The two available Type 3 moves (each between the green and blue curves) switch between two isotopy classes of diagrams, so it suffices to show that each isotopy class is distinct. We use chord diagrams of virtual multistrings (see~\cite{MBM}).

\begin{lem}\label{lem:isotopy}
Two virtual $n$-strings are related by a sequence of isotopies, stabilizations, and destabilizations if and only if the 
chord diagrams are the same.
\end{lem}

\begin{proof}
Let $\A_1$ and $\A_2$ be virtual $n$-strings. Note that isotopies, stabilizations, and destabilizations do not change the relative order of crossings. Hence the chord diagram does not change under such moves and the implication follows.

Now suppose the 
chord diagrams associated to $\A_1$ and $\A_2$ are the same, and consider regular neighborhoods of $\A_1$ and $\A_2$. Since the 
chord diagrams are the same, the regular neighborhoods are isotopic. Capping off the boundaries with disks corresponds to a sequence of destabilizations. Thus the $n$-strings are isotopic on the resulting surface, and so $\A_1$ and $\A_2$ are related by a sequence of isotopies, stabilizations, and destabilizations. 
\end{proof}

Returning to our example, the 
chord diagrams associated to each isotopy class are distinct. Hence, by Lemma~\ref{lem:isotopy}, the associated 3-strings cannot be isotopic. Thus these 3-strings form a counterexample to Kadokami's statement.

\section{Proof of Theorem~\ref{thm:complexity}}

\begin{proof}
Let $\A$ be a connected non-parallel virtual $n$-string on a surface $M$. We first show that a genus-minimal and crossing-minimal representative of $[\A]_V$ exists.

Suppose $\wt{\A}$ is a crossing-minimal representative of $[\A]_V$ on a surface $N$. Using the same argument as in Theorem~\ref{thm:Kadokami}, we apply a sequence of destabilizations to $N$ to obtain a genus-irreducible $n$-string $\wt{\B}$ on a surface $N'$ without increasing the number of crossings. By Theorem~\ref{thm:IMNfull}, $\wt{\B}$ is a genus-minimal representative of $[\A]_V$ and thus the desired representative. Uniqueness up to isotopies and Type 3 moves follows by Theorem~\ref{thm:IMNfull} and Corollary~\ref{cor:Hass}.

We now show that $\A$ can be virtually homotoped to $\wt{\B}$ while monotonically decreasing genus and the number of crossings. By Theorem~\ref{thm:Hass}, we monotonically obtain a representative of $[\A]$ realizing the minimal number of crossings. 
Again using the same argument from the proof of Theorem~\ref{thm:Kadokami}, 
 we apply a sequence of destabilizations to $M$ to obtain a genus-irreducible $n$-string $\B$ on a surface $M'$ without increasing the number of crossings. As $\B$ and $\wt{\B}$ are both genus-irreducible, we obtain a homeomorphism $\phi:M'\to N'$ such that $\phi(\B)$ is homotopic to $\wt{\B}$. By Theorem~\ref{thm:Hass}, we may assume this homotopy monotonically decreases the number of crossings. Hence we obtained $\wt{\B}$ from $\A$ as desired. 
\end{proof}

\section{Concluding Remarks}

\begin{rem}\label{rem:final1}
By itself, Theorem~\ref{thm:Kadokami} does not show that we can monotonically reduce crossings in obtaining the global minimum. Instead, we rely on the result of Hass and Scott. However, without Kadokami's statement, the crossing-irreducible diagram guaranteed by Hass--Scott would not necessarily be crossing-minimal.
\end{rem}


\begin{rem}\label{rem:final3}
Consider a genus-minimal $n$-string $\A$ on a non-connected surface, 
where each 
connected component supports a non-parallel multistring $\A_i$. By Theorem~\ref{thm:complexity} and the genus-minimality of $\A$, we can find a crossing- and genus-minimal representative of $\A_i$, and hence of $\A$. However, in obtaining such a representative of $[\A]_V$, 
we cannot necessarily obtain $\A$ 
by monotonically decreasing genus and the number of crossings.
\end{rem}

\begin{rem}\label{rem:final4}
We conjecture that, for any virtual $n$-string $\A$, there are representatives of $[\A]_V$ that are simultaneously crossing-minimal and genus-minimal. However, our techniques do not apply to $n$-strings not already covered by Theorem~\ref{thm:complexity}. Without Kadokami's statement, we do not even know if crossing-irreducibility and crossing-minimality are necessarily equivalent. Consequently, for arbitrary $n$-strings with $n\geq 3$, it is currently unknown whether a genus-minimal representative must be homotopic to a crossing-irreducible representative.
\end{rem}

{\small
\bibliographystyle{abbrv}
\bibliography{references}
}
\end{document}